\documentclass[10pt,oneside,english]{amsart}
\usepackage[T1]{fontenc}
\usepackage[latin9]{inputenc}
\usepackage{geometry}
\usepackage{amsthm}
\usepackage{amssymb}

\makeatletter
\numberwithin{equation}{section}
\numberwithin{figure}{section}
  \theoremstyle{plain}
  \newtheorem*{thm*}{\protect\theoremname}
\theoremstyle{plain}
\newtheorem{thm}{\protect\theoremname}
  \theoremstyle{definition}
  \newtheorem{example}[thm]{\protect\examplename}
  \theoremstyle{plain}
  \newtheorem{prop}[thm]{\protect\propositionname}
  \theoremstyle{plain}
  \newtheorem{cor}[thm]{\protect\corollaryname}
  \theoremstyle{plain}
  \newtheorem{lem}[thm]{\protect\lemmaname}
  \theoremstyle{remark}
  \newtheorem{rem}[thm]{\protect\remarkname}

\usepackage{xypic}
\newtheorem{parn}{}[subsection]

\makeatother

\usepackage{babel}
  \providecommand{\corollaryname}{Corollary}
  \providecommand{\examplename}{Example}
  \providecommand{\lemmaname}{Lemma}
  \providecommand{\propositionname}{Proposition}
  \providecommand{\remarkname}{Remark}
  \providecommand{\theoremname}{Theorem}
\providecommand{\theoremname}{Theorem}

\begin{document}

\author{Adrien Dubouloz}

\address{Institut de Mathématiques de Bourgogne, 9 avenue Alain Savary, 21
078 Dijon, France}

\email{adrien.dubouloz@u-bourgogne.fr}

\author{Takashi Kishimoto}

\address{Department of Mathematics, Faculty of Science, Saitama University,
Saitama 338-8570, Japan}

\email{tkishimo@rimath.saitama-u.ac.jp}

\thanks{This project was partialy funded by ANR Grant \textquotedbl{}BirPol\textquotedbl{}
ANR-11-JS01-004-01, Grant-in-Aid for Scientific Research of JSPS no.
24740003, and Polish Ministry of Science and Higher Education, \textquotedbl{}Iuventus
Plus\textquotedbl{}, Grant no. IP/2012/038/272. The research was done
during a visit of the first author at the University of Saitama. The
authors thank this institution for its generous support and the excellent
working conditions. }


\title{Families of affine ruled surfaces: existence of cylinders}
\begin{abstract}
We show that the generic fiber of a family $f:X\rightarrow S$ of
smooth $\mathbb{A}^{1}$-ruled affine surfaces always carries an $\mathbb{A}^{1}$-fibration,
possibly after a finite extension of the base $S$. In the particular
case where the general fibers of the family are irrational surfaces,
we establish that up to shrinking $S$, such a family actually factors
through an $\mathbb{A}^{1}$-fibration $\rho:X\rightarrow Y$ over
a certain $S$-scheme $Y\rightarrow S$ induced by the MRC-fibration
of a relative smooth projective model of $X$ over $S$. For affine
threefolds $X$ equipped with a fibration $f:X\rightarrow B$ by irrational
$\mathbb{A}^{1}$-ruled surfaces over a smooth curve $B$, the induced $\mathbb{A}^{1}$-fibration
$\rho:X\rightarrow Y$ can also be obtained from a relative Minimal
Model Program applied to a smooth projective model of $X$ over $B$. 
\end{abstract}
\maketitle

\section*{Introduction}

The general structure of smooth non complete surfaces $X$ with negative
(logarithmic) Kodaira dimension is not fully understood yet. For say
smooth quasi-projective surfaces over an algebraically closed field
of characteristic zero, it was established by Keel and McKernan \cite{KM99}
that the negativity of the Kodaira dimension is equivalent to the
fact that $X$ is generically covered by images of the affine line
$\mathbb{A}^{1}$ in the sense that the set of points $x\in X$ with
the property that there exists a non constant morphism $f:\mathbb{A}^{1}\rightarrow X$
such that $x\in f(\mathbb{A}^{1})$ is dense in $X$ with respect
to the Zariski topology. This property, called $\mathbb{A}^{1}$-uniruledness
is equivalent to the existence of an open embedding $X\hookrightarrow(\overline{X},B)$
into a complete variety $\overline{X}$ covered by proper rational
curves meeting the boundary $B=\overline{X}\setminus X$ in at most
one point. In the case where $X$ is smooth and affine, an earlier
deep result of Miyanishi-Sugie \cite{MS80} asserts the stronger property
that $X$ is $\mathbb{A}^{1}$-ruled: there exists a Zariski dense
open subset $U\subset X$ of the form $U\simeq Z\times\mathbb{A}^{1}$
for a suitable smooth curve $Z$. Equivalently, $X$ admits a surjective
flat morphism $\rho:X\rightarrow C$ to an open subset $C$ of a smooth
projective model $\overline{Z}$ of $Z$, whose generic fiber is isomorphic
to the affine line over the function field of $C$. Such a morphism
$\rho:X\rightarrow C$ is called an $\mathbb{A}^{1}$-fibration, and
we say that $\rho$ is of affine type or complete type when the base
curve $C$ is affine or complete, respectively. 

Smooth $\mathbb{A}^{1}$-uniruled but not $\mathbb{A}^{1}$-ruled
affine varieties are known to exist in every dimension $\geq3$ \cite{DK14}.
Many examples of $\mathbb{A}^{1}$-uniruled affine threefolds can
be constructed in the form of flat families $f:X\rightarrow B$ of
smooth $\mathbb{A}^{1}$-ruled affine surfaces parametrized by a smooth
base curve $B$. For instance, the complement $X$ of a smooth cubic
surface $S\subset\mathbb{P}_{\mathbb{C}}^{3}$ is the total space
of a family $f:X\rightarrow\mathbb{A}^{1}=\mathrm{Spec}(\mathbb{C}[t])$
of $\mathbb{A}^{1}$-ruled surfaces induced by the restriction of
a pencil $\overline{f}:\mathbb{P}^{3}\dashrightarrow\mathbb{P}^{1}$
on $\mathbb{P}^{3}$ generated by $S$ and three times a tangent hyperplane
$H$ to $S$ whose intersection with $S$ consists of a cuspidal cubic
curve. The general fibers of $f$ have negative Kodaira dimension,
carrying $\mathbb{A}^{1}$-fibrations of complete type only, and the
failure of $\mathbb{A}^{1}$-ruledness is intimately related to the
fact that the generic fiber $X_{\eta}$ of $f$, which is a surface
defined over the field $K=\mathbb{C}(t)$, does not admit any $\mathbb{A}^{1}$-fibration
defined over $\mathbb{C}(t)$. Nevertheless, it was noticed in \cite[Theorem 6.1]{GMM14}
that one can infer straight from the construction of $f:X\rightarrow\mathbb{A}^{1}$
the existence a finite base extension $\mathrm{Spec}(L)\rightarrow\mathrm{Spec}(K)$
for which the surface $X_{\eta}\times_{\mathrm{Spec}(K)}\mathrm{Spec}(L)$
carries an $\mathbb{A}^{1}$-fibration $\rho:X_{\eta}\times_{\mathrm{Spec}(K)}\mathrm{Spec}(L)\rightarrow\mathbb{P}_{L}^{1}$
defined over the field $L$. 

A natural question is then to decide whether this phenomenon holds
in general for families $f:X\rightarrow B$ of $\mathbb{A}^{1}$-ruled
affine surfaces parameterized by a smooth base curve $B$, namely,
does the existence of $\mathbb{A}^{1}$-fibrations on the general
fibers of $f$ imply the existence of one on the generic fiber of
$f$, possibly after a finite extension of the base $B$ ? A partial
positive answer is given by Gurjar-Masuda-Miyanishi in \cite[Theorem 3.8]{GMM14}
under the additional assumption that the general fibers of $f$ carry
$\mathbb{A}^{1}$-fibrations of affine type. The main result in \emph{loc.
cit.} is derived from the study of log-deformations of suitable relative
normal projective models $\overline{f}:(\overline{X},D)\rightarrow B$
of $X$ over $B$ with appropriate boundaries $D$. It is established
in particular that the structure of the boundary divisor of a well
chosen smooth projective completion of a general closed fiber $X_{s}$
is stable under small deformations, a property which implies in turn,
possibly after a finite extension of the base $B$, the existence
of an $\mathbb{A}^{1}$-fibration of affine type on the generic fiber
of $f$.  This log-deformation theoretic approach is also central
in the related recent work of Flenner-Kaliman-Zaidenberg \cite{FKZ13}
on the classification of normal affine surfaces with $\mathbb{A}^{1}$-fibrations
of affine type up to a certain notion of deformation equivalence,
defined for families which admit suitable relative projective models
satisfying Kamawata's axioms of logarithmic deformations of pairs
\cite{Ka78}. The fact that the $\mathbb{A}^{1}$-fibrations under
consideration are of affine type plays again a crucial role and, in
contrast with the situation considered in \cite{GMM14}, the restrictions
imposed on the families imply the existence of $\mathbb{A}^{1}$-fibrations
of affine type on their generic fibers.\\

Our main result (Theorem \ref{thm:MainTh-Family-Etale-cylinder})
consists of a generalization of the results in \cite{GMM14} to families
$f:X\rightarrow S$ of $\mathbb{A}^{1}$-ruled surfaces over an arbitrary
normal base $S$, which also includes the case where a general closed
fiber $X_{s}$ of $f$ admits $\mathbb{A}^{1}$-fibrations of complete
type only. In particular, we obtain the following positive answer
to Conjecture 6.2 in \cite{GMM14}: 
\begin{thm*}
Let $f:X\rightarrow S$ be dominant morphism between normal complex
algebraic varieties whose general fibers are smooth $\mathbb{A}^{1}$-ruled
affine surfaces. Then there exists a dense open subset $S_{*}\subset S$,
a finite \'etale morphism $T\rightarrow S_{*}$ and a normal $T$-scheme
$h:Y\rightarrow T$ such that the induced morphism $f_{T}=\mathrm{p}r_{T}:X_{T}=X\times_{S_{*}}T\rightarrow T$
factors as 
\[
f_{T}=h\circ\rho:X_{T}\stackrel{\rho}{\longrightarrow}Y\stackrel{h}{\longrightarrow}T,
\]
where $\rho:X_{T}\rightarrow Y$ is an $\mathbb{A}^{1}$-fibration. 
\end{thm*}
In contrast with the log-deformation theoretic strategy used in \cite{GMM14},
which involves the study of certain Hilbert schemes of rational curves
on well-chosen relative normal projective models $\overline{f}:(\overline{X},B)\rightarrow S$
of $X$ over $S$, our approach is more elementary, based on the notion
of Kodaira dimension \cite{Ii77} adapted to the case of geometrically
connected varieties defined over arbitrary base fields of characteristic
zero. Indeed, the hypothesis means equivalently that the general fibers
of $f$ have negative Kodaira dimension. This property is in turn
inherited by the generic fiber of $f$, which is a smooth affine surface
defined over the function field of $S$, thanks to a standard Lefschetz
principle argument. Then we are left with checking that a smooth affine
surface $X$ defined over an arbitrary base field $k$ of characteristic
zero and with negative Kodaira dimension admits an $\mathbb{A}^{1}$-fibration,
possibly after a suitable finite base extension $\mathrm{Spec}(k_{0})\rightarrow\mathrm{Spec}(k)$,
a fact which follows immediately from finite type hypotheses and the
aforementioned characterization of Miyanishi-Sugie \cite{MS80}. \\

The article is organized as follows. The first section contains a
review of the structure of smooth affine surfaces of negative Kodaira
dimension over arbitrary base fields $k$ of characteristic zero.
We show in particular that every such surface $X$ admits an $\mathbb{A}^{1}$-fibration
after a finite extension of the base field $k$, and we give criteria
for the existence of $\mathbb{A}^{1}$-fibrations defined over $k$.
These results are then applied in the second section to the study
of deformations $f:X\rightarrow S$ of smooth $\mathbb{A}^{1}$-ruled
affine surfaces: after giving the proof of the main result, Theorem
\ref{thm:MainTh-Family-Etale-cylinder}, we consider in more detail
the particular situation where the general fibers of $f:X\rightarrow S$
are irrational. In this case, after shrinking $S$ if necessary, we
show that the morphism $f$ actually factors through an $\mathbb{A}^{1}$-fibration
$\rho:X\rightarrow Y$ over an $S$-scheme $h:Y\rightarrow S$ which
coincides, up to birational equivalence, with the Maximally Rationally
Connected quotient of a relative smooth projective model $\overline{f}:\overline{X}\rightarrow S$
of $X$ over $S$. The last section is devoted to the case of affine
threefolds equipped with a fibration $f:X\rightarrow B$ by irrational
$\mathbb{A}^{1}$-ruled surfaces over a smooth base curve $B$: we
explain in particular how to construct an $\mathbb{A}^{1}$-fibration
$\rho:X\rightarrow Y$ factoring $f$ by means of a relative Minimal
Model Program applied to a smooth projective model $\overline{f}:\overline{X}\rightarrow B$
of $X$ over $B$.

\section{$\mathbb{A}^{1}$-ruledness of affine surfaces over non closed field }

\subsection{Logarithmic Kodaira dimension }

\begin{parn} \label{par:Kod-dim} Let $X$ be a smooth geometrically
connected algebraic variety defined over a field $k$ of characteristic
zero. By virtue of Nagata compactification \cite{Na62} and Hironaka
desingularization \cite{Hi64} theorems, there exists an open immersion
$X\hookrightarrow(\overline{X},B)$ into a smooth complete algebraic
variety $\overline{X}$ with reduced SNC boundary divisor $B=\overline{X}\setminus X$.
The (logarithmic) Kodaira dimension $\kappa(X)$ of $X$ is then defined
as the Iitaka dimension \cite{Ii70} of the pair $(\overline{X};\omega_{\overline{X}}(\log B))$
where $\omega_{\overline{X}}(\log B)=(\det\Omega_{\overline{X}/k}^{1})\otimes\mathcal{O}_{\overline{X}}(B)$.
So letting $\mathcal{R}(\overline{X},B)=\bigoplus_{m\geq0}H^{0}(\overline{X},\omega_{\overline{X}}(\log B)^{\otimes m})$,
we have $\kappa(X)=\mathrm{tr}.\deg_{k}\mathcal{R}(\overline{X},B)-1$
if $H^{0}(\overline{X},\omega_{\overline{X}}(\log B)^{\otimes m})\neq0$
for sufficiently large $m$. Otherwise, if $H^{0}(\overline{X},\omega_{\overline{X}}(\log B)^{\otimes m})=0$
for every $m\geq1$, we set by convention $\kappa(X)=-\infty$ and
we say that $\kappa(X)$ is negative. The so-defined element $\kappa(X)\in\left\{ 0,\ldots,\mathrm{dim}_{k}X\right\} \cup\{-\infty\}$
is independent of the choice of a smooth complete model $(\overline{X},B)$
\cite{Ii77}. 

Furthermore, the Kodaira dimension of $X$ is invariant under arbitrary
extensions of the base field $k$. Indeed, given an extension $k\subset k'$,
the pair $(\overline{X}_{k'},B_{k'})$ obtained by the base change
$\mathrm{Spec}(k')\rightarrow\mathrm{Spec}(k)$ is a smooth complete
model of $X_{k'}=X\times_{\mathrm{Spec}(k)}\mathrm{Spec}(k')$ with
reduced SNC boundary $B_{k'}$. Furthermore letting $\pi:\overline{X}_{k'}\rightarrow\overline{X}$
be the corresponding faithfully flat morphism, we have $\omega_{\overline{X}_{k'}}(\log B_{k'})\simeq\pi^{*}\omega_{X}(\log B)$
and so $\mathcal{R}(X_{k'})\simeq\mathcal{R}(X)\otimes_{k}k'$ by
the flat base change theorem \cite[Proposition III.9.3]{Ha77}. Thus
$\kappa(X)=\kappa(X_{k'})$. 

\end{parn}
\begin{example}
\label{Exa:forms-affine-line} The affine line $\mathbb{A}_{k}^{1}$
is the only smooth geometrically connected non complete curve $C$
with negative Kodaira dimension. Indeed, let $\overline{C}$ be a
smooth projective model of $C$ and let $\overline{C}_{\overline{k}}$
be the curve obtained by the base change to an algebraic closure $\overline{k}$
of $k$. Since $C$ is non complete, $B=\overline{C}_{\overline{k}}\setminus C_{\overline{k}}$
consists of a finite collection of closed points $p_{1},\ldots p_{s}$,
$s\geq1$, on which the Galois group $\mathrm{Gal}(\overline{k}/k)$
acts by $k$-automorphisms of $\overline{C}_{\overline{k}}$. Clearly,
$H^{0}(\overline{C}_{\overline{k}},\omega_{\bar{C}_{\overline{k}}}(\log B)^{\otimes m})\neq0$
unless $\overline{C}_{\overline{k}}\simeq\mathbb{P}_{\overline{k}}^{1}$
and $s=1$. Since $p_{1}$ is then necessarily $\mathrm{Gal}(\overline{k}/k)$-invariant,
$\overline{C}\setminus C$ consists of unique $k$-rational point,
showing that $\overline{C}\simeq\mathbb{P}_{k}^{1}$ and $C\simeq\mathbb{A}_{k}^{1}$. 
\end{example}

\subsection{Smooth affine surfaces with negative Kodaira dimension}

\indent\newline\indent Recall that by virtue of \cite{MS80}, a smooth
affine surface $X$ defined over an algebraically closed field of
characteristic zero has negative Kodaira dimension if and only if
it is $\mathbb{A}^{1}$-ruled: there exists a Zariski dense open subset
$U\subset X$ of the form $U\simeq Z\times\mathbb{A}^{1}$ for a suitable
smooth curve $Z$. In fact, the projection $\mathrm{pr}_{Z}:U\simeq Z\times\mathbb{A}^{1}\rightarrow Z$
always extends to an $\mathbb{A}^{1}$-fibration $\rho:X\rightarrow C$
over an open subset $C$ of a smooth projective model $\overline{Z}$
of $Z$. This characterization admits the following straightforward
generalization to arbitrary base fields of characteristic zero: 
\begin{thm}
\label{prop:Surface-etale-cylinder} Let $X$ be a smooth geometrically
connected affine surface defined over a field $k$ of characteristic
zero. Then the following are equivalent:

a) The Kodaira dimension $\kappa(X)$ of $X$ is negative.

b) For some finite extension $k_{0}$ of $k$, the surface $X_{k_{0}}$
contains an open subset $U\simeq Z\times\mathbb{A}_{k_{0}}^{1}$ for
some smooth curve $Z$ defined over $k_{0}$.

c) There exists a finite extension $k_{0}$ of $k$ and an $\mathbb{A}^{1}$-fibration
$\rho:X_{k_{0}}\rightarrow C_{0}$ over a smooth curve $C_{0}$ defined
over $k_{0}$. \end{thm}
\begin{proof}
Clearly c) implies b) and b) implies a). To show that a) implies c),
we observe that letting $\overline{k}$ be an algebraic closure of
$k$, we have $\kappa(X_{\overline{k}})=\kappa(X)<0$. It then follows
from the aforementioned result of Miyanishi-Sugie \cite{MS80} that
$X_{\overline{k}}$ admits an $\mathbb{A}^{1}$-fibration $q:X_{\overline{k}}\rightarrow C$
over a smooth curve $C$, with smooth projective model $\overline{C}$.
Since $X_{\overline{k}}$ and $\overline{C}$ are of finite type over
$\overline{k}$, there exists a finite extension $k\subset k_{0}$
such that $q:X_{\overline{k}}\rightarrow\overline{C}$ is obtained
from a morphism $\rho:X_{k_{0}}\rightarrow\overline{C}_{0}$ to a
smooth projective curve $\overline{C}_{0}$ defined over $k_{0}$
by the base extension $\mathrm{\mathrm{Spec}}(\overline{k})\rightarrow\mathrm{Spec}(k_{0})$.
By virtue of Example \ref{Exa:forms-affine-line}, $\rho:X_{k_{0}}\rightarrow\overline{C}_{0}$
is an $\mathbb{A}^{1}$-fibration. 
\end{proof}
\noindent Examples of smooth affine surfaces $X$ of negative Kodaira
dimension without any $\mathbb{A}^{1}$-fibration defined over the
base field but admitting $\mathbb{A}^{1}$-fibrations of complete
type after a finite base extension were already constructed in \cite{DK14}.
The following example illustrates the fact that a similar phenomenon
occurs for $\mathbb{A}^{1}$-fibrations of affine type, providing
in particular a negative answer to Problem 3.13 in \cite{GMM14}. 
\begin{example}
\label{Ex:A1-fib-extension} Let $B\subset\mathbb{P}_{k}^{2}=\mathrm{Proj}(k[x,y,z])$
be a smooth conic without $k$-rational points defined by a quadratic
form $q=x^{2}+ay^{2}+bz^{2}$, where $a,b\in k^{*}$, and let $\overline{X}\subset\mathbb{P}_{k}^{3}=\mathrm{Proj}(k[x,y,z,t])$
be the smooth quadric surface defined by the equation $q(x,y,z)-t^{2}=0$.
The complement $X\subset\overline{X}$ of the hyperplane section $\left\{ t=0\right\} $
is a $k$-rational smooth affine surface with $\kappa(X)<0$, which
does not admit any $\mathbb{A}^{1}$-fibration $\rho:X\rightarrow C$
over a smooth, affine or projective curve $C$. Indeed, if such a
fibration existed then a smooth projective model of $C$ would be
isomorphic to $\mathbb{P}_{k}^{1}$; since the fiber of $\rho$ over
a general $k$-rational point of $C$ is isomorphic to $\mathbb{A}_{k}^{1}$,
its closure in $\overline{X}$ would intersect the boundary $\overline{X}\setminus X\simeq B$
in a unique point, necessarily $k$-rational, in contradiction with
the choice of $B$. 

In constrast, for a suitable finite extension $k\subset k'$, the
surface $X_{k'}$ becomes isomorphic to the complement of the diagonal
in $\overline{X}_{k'}\simeq\mathbb{P}_{k'}^{1}\times\mathbb{P}_{k'}^{1}$
and hence, it admits at least two distinct $\mathbb{A}^{1}$-fibrations
over $\mathbb{P}_{k'}^{1}$, induced by the restriction of the first
and second projections from $\overline{X}_{k'}$. Furthermore, since
$X_{k'}$ is isomorphic to the smooth affine quadric in $\mathbb{A}_{k'}^{3}=\mathrm{Spec}(k'[u,v,w])$
with equation $uv-w^{2}=1$, it also admits two distinct $\mathbb{A}^{1}$-fibrations
over $\mathbb{A}_{k'}^{1}$, induced by the restrictions of the projections
$\mathrm{pr}_{u}$ and $\mathrm{pr}_{v}$. 
\end{example}

\subsection{Existence of $\mathbb{A}^{1}$-fibrations defined over the base field }

\begin{parn} \label{par:galois-descent} The previous example illustrates
the general fact that if $X$ is a smooth geometrically connected
affine surface with $\kappa(X)<0$ which does not admit any $\mathbb{A}^{1}$-fibration,
then there exists a finite extension $k'$ of $k$ such that $X_{k'}$
admits at least two $\mathbb{A}^{1}$-fibrations of the same type,
either affine or complete, with distinct general fibers. Indeed, by
virtue of Theorem \ref{prop:Surface-etale-cylinder}, there exists
a finite extension $k_{0}$ of $k$ such that $X_{k_{0}}$ admits
an $\mathbb{A}^{1}$-fibration $\rho:X_{k_{0}}\rightarrow C$. Let
$k'$ be the Galois closure of $k_{0}$ in an algebraic closure of
$k$ and let $\rho_{k'}:X_{k'}\rightarrow C_{k'}$ be the $\mathbb{A}^{1}$-fibration
deduced from $\rho$. If $\rho_{k'}:X_{k'}\rightarrow C_{k'}$ is
globally invariant under the action of the Galois group $\mathrm{Gal}(k'/k)$
on $X_{k'}$, in the sense that for every $\Phi\in\mathrm{Gal}(k'/k)$
considered as a Galois automorphism of $X_{k'}$ there exists a commutative
diagram 
\begin{eqnarray*}
X_{k'} & \stackrel{\Phi}{\rightarrow} & X_{k'}\\
\rho_{k'}\downarrow &  & \downarrow\rho_{k'}\\
C_{k'} & \stackrel{\varphi}{\rightarrow} & C_{k'}
\end{eqnarray*}
for a certain $k'$-automorphism $\varphi$ of $C_{k'}$, then we
would obtain a Galois action of $\mathrm{Gal}(k'/k)$ on $C_{k'}$
for which $\rho_{k'}:X_{k'}\rightarrow C_{k'}$ becomes an equivariant
morphism. Since $C_{k'}$ is quasi-projective and $\rho_{k}'$ is
affine, it would follow from Galois descent that there exists a curve
$\tilde{C}$ defined over $k$ and a morphism $q:X\rightarrow\tilde{C}$
defined over $k$ such that $\rho_{k'}:X_{k'}\rightarrow C_{k'}$
is obtained from $q$ by the base change $\mathrm{Spec}(k')\rightarrow\mathrm{Spec}(k)$.
Since by virtue of Example \ref{Exa:forms-affine-line} the affine
line does not have any nontrivial form, the generic fiber of $q$
would be isomorphic to the affine line over the field of rational
functions of $\tilde{C}$ and so, $q:X\rightarrow\tilde{C}$ would
be an $\mathbb{A}^{1}$-fibration defined over $k$, in contradiction
with our hypothesis. So there exists at least an element $\Phi\in\mathrm{Gal}(k'/k)$
considered as a $k$-automorphism of $X_{k'}$ such that the $\mathbb{A}^{1}$-fibrations
$\rho_{k'}:X_{k'}\rightarrow C_{k'}$ and $\rho_{k'}\circ\varphi:X_{k'}\rightarrow C_{k'}$
have distinct general fibers. 

\end{parn}

\noindent Arguing backward, we obtain the following useful criterion: 
\begin{prop}
Let $X$ be a smooth geometrically connected affine surface with $\kappa(X)<0$.
If there exists a finite Galois extension $k'$ of $k$ such that
$X_{k'}$ admits a unique $\mathbb{A}^{1}$-fibration $\rho':X_{k'}\rightarrow C_{k'}$
up to composition by automorphisms of $C_{k'}$, then $\rho':X_{k'}\rightarrow C_{k'}$
is obtained by base extension from an $\mathbb{A}^{1}$-fibration
$\rho:X\rightarrow C$ defined over $k$. 
\end{prop}
\noindent
\begin{cor}
\label{cor:Irrational-fibration} A smooth geometrically connected
irrational affine surface $X$ has negative Kodaira dimension if and
only if it admits an $\mathbb{A}^{1}$-fibration $\rho:X\rightarrow C$
over a smooth irrational curve $C$ defined over the base field $k$.
Furthermore for every extension $k'$ of $k$, $\rho_{k'}:X_{k'}\rightarrow C_{k'}$
is the unique $\mathbb{A}^{1}$-fibration on $X_{k'}$ up to composition
by automorphisms of $C_{k'}$. \end{cor}
\begin{proof}
Uniqueness is clear since otherwise $C_{k'}$ would be dominated by
a general fiber of another $\mathbb{A}^{1}$-fibration on $X_{k'}$,
and hence would be rational, implying in turn the rationality of $X$.
By virtue of Theorem \ref{prop:Surface-etale-cylinder}, there exists
a finite Galois extension $k'$ of $k$ and an $\mathbb{A}^{1}$-fibration
$\rho':X_{k'}\rightarrow C'$ over a smooth curve $C'$. The latter
is irrational as $X$ is irrational, which implies that $\rho':X_{k'}\rightarrow C'$
is the unique $\mathbb{A}^{1}$-fibration on $X_{k'}$. So $\rho'$
descend to an $\mathbb{A}^{1}$-fibration $\rho:X\rightarrow C$ over
a smooth irrational curve $C$ defined over $k$. 
\end{proof}

\section{Families of $\mathbb{A}^{1}$-ruled affine surfaces }

\subsection{Existence of \'etale $\mathbb{A}^{1}$-cylinders }

This subsection is devoted to the proof of the following:
\begin{thm}
\label{thm:MainTh-Family-Etale-cylinder} Let $X$ and $S$ be normal
algebraic varieties defined over a field $k$ of infinite transcendence
degree over $\mathbb{Q}$, and let $f:X\rightarrow S$ be a dominant
affine morphism with the property that for a general closed point
$s\in S$, the fiber $X_{s}$ is a smooth geometrically connected
affine surface with negative Kodaira dimension. Then there exists
an open subset $S_{*}\subset S$, a finite \'etale morphism $T\rightarrow S_{*}$
and a normal $T$-scheme $h:Y\rightarrow T$ such that $f_{T}=\mathrm{p}r_{T}:X_{T}=X\times_{S_{*}}T\rightarrow T$
factors as 
\[
f_{T}=h\circ\rho:X_{T}\stackrel{\rho}{\longrightarrow}Y\stackrel{h}{\longrightarrow}T
\]
where $\rho:X_{T}\rightarrow Y$ is an $\mathbb{A}^{1}$-fibration. \end{thm}
\begin{proof}
Shrinking $S$ if necessary, we may assume that $S$ is affine, that
$f:X\rightarrow S$ is smooth and that $\kappa(X_{s})<0$ for every
closed point $s\in S$. It is enough to show that the fiber $X_{\eta}$
of $f$ over the generic point $\eta$ of $S$ is geometrically connected,
with negative Kodaira dimension. Indeed, if so, then by Theorem \ref{prop:Surface-etale-cylinder}
above, there exists a finite extension $L$ of $K=\mathrm{Frac}(\Gamma(S,\mathcal{O}_{S}))$
and an $\mathbb{A}^{1}$-fibration $\rho:X_{\eta}\times_{\mathrm{Spec}(K)}\mathrm{Spec}(L)\rightarrow C$
onto a smooth curve $C$ defined over $L$. Letting $T$ be the normalization
of $S$ in $L$ and shrinking $T$ again if necessary, we obtain a
finite \'etale morphism $T\rightarrow S$ such that the generic fiber
of $\mathrm{pr}_{T}:X_{T}\rightarrow T$ is isomorphic to the $\mathbb{A}^{1}$-fibered
surface $\rho:X_{\eta}\times_{\mathrm{Spec}(K)}\mathrm{Spec}(L)\rightarrow C$
and then the assertion follows from Lemma \ref{lem:Generic-to-openset}
below. 

Since $X$ and $S$ are affine and of finite type over $k$, there
exists a subfield $k_{0}$ of $k$ of finite transcendence degree
over $\mathbb{Q}$, and a smooth morphism $f_{0}:X_{0}\rightarrow S_{0}$
of $k_{0}$-varieties such that $f:X\rightarrow S$ is obtained from
$f_{0}:X_{0}\rightarrow S_{0}$ by the base extension $\mathrm{Spec}\left(k\right)\rightarrow\mathrm{Spec}(k_{0})$.
The field $K_{0}=\mathrm{Frac}(\Gamma(S_{0},\mathcal{O}_{S_{0}}))$
has finite transcendence degree over $\mathbb{Q}$ and hence, it admits
a $k_{0}$-embedding $\xi:K_{0}\hookrightarrow k$. Letting $(X_{0})_{\eta_{0}}$
be the fiber of $f_{0}$ over the generic point $\eta_{0}:\mathrm{Spec}(K_{0})\rightarrow S_{0}$
of $S_{0}$, the composition $\Gamma(S_{0},\mathcal{O}_{S_{0}})\hookrightarrow K_{0}\hookrightarrow k$
induces a $k$-homomorphism $\Gamma(S_{0},\mathcal{O}_{S_{0}})\otimes_{k_{0}}k\rightarrow k$
defining a closed point $s:\mathrm{Spec}(k)\rightarrow\mathrm{Spec}(\Gamma(S_{0},\mathcal{O}_{S_{0}})\otimes_{k_{0}}k)=S$
of $S$ for which obtain the following commutative diagram 

\[\xymatrix@!=16pt{ & X_s \ar[dl] \ar@{->}'[d][dd] \ar[rr] & & X \ar[dl] \ar[dd]_(0.5){f} \\ (X_0)_{\eta_0} \ar[rr] \ar[dd] & & X_0 \ar[dd]_(0.4){f_0} \\ & \mathrm{Spec}(k) \ar[dl]_{\xi^*} \ar@{->}'[r][rr]^{s} & & S \ar[r] \ar[dl]  & \mathrm{Spec}(k) \ar[dl] \\ \mathrm{Spec}(K_0) \ar[rr]^{\eta_0} & & S_0 \ar[r]  & \mathrm{Spec}(k_0) &  }\]  The
bottom square of the cube being cartesian by construction, we deduce
that 
\[
(X_{0})_{\eta_{0}}\times_{\mathrm{Spec}(K_{0})}\mathrm{Spec}(k)\simeq X_{0}\times_{S_{0}}\mathrm{Spec}(k)\simeq X\times_{S}\mathrm{Spec}(k)=X_{s}.
\]
Since by assumption, $X_{s}$ is geometrically connected with $\kappa(X_{s})<0$,
we conclude that $(X_{0})_{\eta_{0}}$ is geometrically connected
and has negative Kodaira dimension. This implies in turn that $X_{\eta}$
is geometrically connected and that $\kappa(X_{\eta})<0$ as desired. 
\end{proof}
In the proof of the above theorem, we used the following lemma:
\begin{lem}
\label{lem:Generic-to-openset} Let $f:X\rightarrow S$ be a dominant
affine morphism between normal varieties defined over a field $k$
of characteristic zero. Then the following are equivalent:

a) The generic fiber $X_{\eta}$ of $f$ admits an $\mathbb{A}^{1}$-fibration
$q:X_{\eta}\rightarrow C$ over a smooth curve $C$ defined over the
fraction field $K$ of $S$.

b) There exists an open subset $S_{*}$ of $S$ and a normal $S_{*}$-scheme
$h:Y\rightarrow S_{*}$ of relative dimension $1$ such that the restriction
of $f$ to $V=f^{-1}(S_{*})$ factors as $f\mid_{V}=h\circ\rho:V\rightarrow Y\rightarrow S_{*}$
where $\rho:V\rightarrow Y$ is an $\mathbb{A}^{1}$-fibration. \end{lem}
\begin{proof}
If b) holds then letting $L$ be the fraction field of $Y$, we have
a commutative diagram
\[
\begin{array}{ccccc}
V_{\xi}=X_{\xi} & \rightarrow & V_{\eta}=X_{\eta} & \rightarrow & V\\
\downarrow\rho_{\xi} &  & \downarrow\rho_{\eta} &  & \downarrow\rho\\
\mathrm{Spec}(L) & \stackrel{\xi}{\rightarrow} & C=Y_{\eta} & \rightarrow & Y\\
 &  & \downarrow h_{\eta} &  & \downarrow h\\
 &  & \mathrm{Spec}(K) & \stackrel{\eta}{\rightarrow} & S_{*}
\end{array}
\]
in which each square is cartesian. It follows that $h_{\eta}:C\rightarrow\mathrm{Spec}(K)$
is a normal whence smooth curve defined over $K$ and that $\rho_{\eta}:X_{\eta}\rightarrow C$
is an $\mathbb{A}^{1}$-fibration. Conversely, suppose that $X_{\eta}$
admits an $\mathbb{A}^{1}$-fibration $q:X_{\eta}\rightarrow C$ and
let $\overline{C}$ be a smooth projective model of $C$ over $K$.
Then there exists an open subset $S_{0}$ of $S$ and a projective
$S_{0}$-scheme $h:Y\rightarrow S_{0}$ whose generic fiber is isomorphic
to $\overline{C}$. After shrinking $S_{0}$ if necessary, the rational
map $\rho:V\dashrightarrow Y$ of $S_{0}$-schemes induced by $q$
becomes a morphism and we obtain a factorization $f\mid_{V}=h\circ\rho$.
By construction, the generic fiber $V_{\xi}$ of $\rho:V\rightarrow Y$
is isomorphic to $V\times_{Y}\mathrm{Spec}(L)\simeq(V\times_{Y}C)\times_{C}\mathrm{Spec}(L)\simeq X_{\eta}\times_{C}\mathrm{Spec}(L)\simeq\mathbb{A}_{L}^{1}$
since $V\times_{Y}C\simeq V_{\eta}\simeq X_{\eta}$ and $\rho:X_{\eta}\rightarrow C\hookrightarrow\overline{C}$
is an $\mathbb{A}^{1}$-fibration. So $\rho:V\rightarrow Y$ is an
$\mathbb{A}^{1}$-fibration. \end{proof}
\begin{example}
\label{ex:delPezzo} Let $R=\mathbb{C}[s^{\pm1},t^{\pm1}]$, $S=\mathrm{Spec}(R)$
and let $D$ be the relatively ample divisor in $\mathbb{P}_{S}^{2}=\mathrm{Proj}_{R}(R[x,y,z])$
defined by the equation $x^{2}+sy^{2}+tz^{2}=0$. The restriction
$h:X=\mathbb{P}_{S}^{2}\setminus D\rightarrow S$ of the structure
morphism defines a family of smooth affine surfaces with the property
that for every closed point $s\in S$, $X_{s}$ is isomorphic to the
complement in $\mathbb{P}_{\mathbb{C}}^{2}$ of the smooth conic $D_{s}$.
In particular $X_{s}$ admits a continuum of pairwise distinct $\mathbb{A}^{1}$-fibrations
$X_{s}\rightarrow\mathbb{A}_{\mathbb{C}}^{1}$, induced by the restrictions
to $X_{s}$ of the rational pencils on $\mathbb{P}_{\mathbb{C}}^{2}$
generated by $D_{s}$ and twice its tangent line at an arbitrary closed
point $p_{s}\in D_{s}$. On the other hand the fiber of $D$ over
the generic point $\eta$ of $S$ is a conic without $\mathbb{C}(s,t)$-rational
point in $\mathbb{P}_{\mathbb{C}(s,t)}^{2}$ and hence, we conclude
by a similar argument as in Example \ref{Ex:A1-fib-extension} that
$X_{\eta}$ does not admit any $\mathbb{A}^{1}$-fibration defined
over $\mathbb{C}(s,t)$. Therefore there is no open subset $S_{*}$
of $S$ over which $h$ can be factored through an $\mathbb{A}^{1}$-fibration. 
\end{example}

\subsection{Deformations of irrational $\mathbb{A}^{1}$-ruled affine surfaces }

\indent\newline\indent In this subsection, we consider the particular
situation of a flat family $f:X\rightarrow S$ over a normal variety
$S$ whose general fibers are irrational $\mathbb{A}^{1}$-ruled affine
surfaces. A combination of Corollary \ref{cor:Irrational-fibration}
and Theorem \ref{thm:MainTh-Family-Etale-cylinder} above implies
that if $f:X\rightarrow S$ is smooth and defined over a field of
infinite transcendence degree over $\mathbb{Q}$, then the generic
fiber $X_{\eta}$ of $f$ is $\mathbb{A}^{1}$-ruled. Equivalently,
there exists an open subset $S_{*}\subset S$ and a normal $S_{*}$-scheme
$h:Y\rightarrow S_{*}$ such that the restriction of $f$ to $X_{*}=X\times_{S}S_{*}$
factors through an $\mathbb{A}^{1}$-fibration $\rho:X_{*}\rightarrow Y$
(see Lemma \ref{lem:Generic-to-openset}). The restriction of $\rho$
to the fiber of $f$ over a general closed point $s\in S_{0}$ is
an $\mathbb{A}^{1}$-fibration $\rho_{s}:X_{s}\rightarrow Y_{s}$
over the normal, whence smooth, curve $Y_{s}$. Since $X_{s}$ is
irrational, $Y_{s}$ is irrational, and so $\rho_{s}:X_{s}\rightarrow Y_{s}$
is the unique $\mathbb{A}^{1}$-fibration on $X_{s}$ up to composition
by automorphisms of $Y_{s}$. So in this case, we can identify $\rho_{s}:X_{s}\rightarrow Y_{s}$
with the Maximally Rationally Connected fibration (MRC-fibration)
$\varphi:\overline{X}_{s}\dashrightarrow Y_{s}$ of a smooth projective
model $\overline{X}_{s}$ of $X_{s}$ in the sense of \cite[IV.5]{Ko96}:
recall that $\varphi$ is unique, characterized by the property that
its general fibers are rationally connected and that for a very general
point $y\in Y_{s}$ any rational curve in $\overline{X}_{s}$ which
meets $\overline{X}_{y}$ is actually contained in $\overline{X}_{y}$.
The $\mathbb{A}^{1}$-fibration $\rho:X_{*}\rightarrow Y$ can therefore
be re-interpreted as being the MRC-fibration of a relative smooth
projective model $\overline{X}$ of $X$ over $S$. 

Reversing the argument, general existence and uniqueness results for
MRC-fibrations allow actually to get rid of the smoothness hypothesis
of a general fiber of $f:X\rightarrow S$ and to extend the conclusion
of Theorem \ref{thm:MainTh-Family-Etale-cylinder} to arbitrary base
fields of characteristic zero. Namely, we obtain the following characterization: 
\begin{thm}
\label{thm:Mrc-fibration}Let $X$ and $S$ be normal varieties defined
over a field $k$ of characteristic zero and let $f:X\rightarrow S$
be a dominant affine morphism with the property that for a general
closed point $s\in S$, the fiber $X_{s}$ is irrational and $\mathbb{A}^{1}$-ruled.
Then there exists an open subset $S_{*}$ and a normal $S_{*}$-scheme
$h:Y\rightarrow S_{*}$ such that the restriction of $f$ to $X_{*}=X\times_{S}S_{*}$
factors as 
\[
f\mid_{X_{*}}=h\circ\rho:X_{*}\stackrel{\rho}{\longrightarrow}Y\stackrel{h}{\longrightarrow}S_{*}
\]
where $\rho:X_{*}\rightarrow Y$ is an $\mathbb{A}^{1}$-fibration. \end{thm}
\begin{proof}
Shrinking $S$ if necessary, we may assume that for every closed point
$s\in S$, $X_{s}$ is irrational and $\mathbb{A}^{1}$-ruled, hence
carrying a unique $\mathbb{A}^{1}$-fibration $\pi_{s}:X_{s}\rightarrow C_{s}$
over an irrational normal curve $C_{s}$. Since $f:X\rightarrow S$
is affine, there exists a normal projective $S$-scheme $\overline{X}\rightarrow S$
and an open embedding $X\hookrightarrow\overline{X}$ of schemes over
$S$. Letting $W\rightarrow\overline{X}$ be a resolution of the singularities
of $\overline{X}$, we may assume up to shrinking $S$ again if necessary
that $W\rightarrow S$ is a smooth morphism. We let $j:X\dashrightarrow W$
be the birational map of $S$-schemes induced by the embedding $X\hookrightarrow\overline{X}$.
By virtue of \cite[Theorem 5.9]{Ko96}, there exists an open subset
$W'$ of $W$, an $S$-scheme $h:Z\rightarrow S$ and a proper morphism
$\overline{q}:W'\rightarrow Z$ such that for every $s\in S$, the
induced rational map $\overline{q}_{s}:W_{s}\dashrightarrow Z_{s}$
is the MRC-fibration for $W_{s}$. On the other hand, since $W_{s}$
is a smooth projective model of $X_{s}$, the induced rational map
$\pi_{s}:\overline{X}_{s}\dashrightarrow C_{s}$ is the MRC-fibration
for $W_{s}$. Consequently, for a general closed point $z\in Z$ with
$h(z)=s$, the fiber $W_{z}$ of $\overline{q}_{s}$, which is an
irreducible proper rational curve contained in $W_{s}$, must coincide
with the closure of the image by $j$ of a general closed fiber of
$\pi_{s}$. The latter being isomorphic to the affine line $\mathbb{A}_{\kappa}^{1}$
over the residue field $\kappa$ of the corresponding point of $C_{s}$,
we conclude that there exists an affine open subset $U$ of $X$ on
which the composition $\overline{q}\circ j:U\rightarrow Z$ is a well
defined morphism with general closed fibers isomorphic to affine lines
over the corresponding residue fields. So $\overline{q}\circ j:U\rightarrow Z$
is an $\mathbb{A}^{1}$-fibration bu virtue of \cite{KM78}. The generic
fiber of $f:X\rightarrow S$ is thus $\mathbb{A}^{1}$-ruled and the
assertion follows from Lemma \ref{lem:Generic-to-openset} above. \end{proof}
\begin{example}
\label{ex:fibration-from-moduli} Let $h:Y\rightarrow S$ be smooth
family of complex projective curves of genus $g\geq2$ over a normal
affine base $S$ et let $\mathcal{T}_{Y/S}$ be the relative tangent
sheaf of $h$. Since by Riemman-Roch $H^{0}(Y_{s},\mathcal{T}_{Y/S,s})=0$
and $\dim H^{1}(Y_{s},\mathcal{T}_{Y/S,s})=g-1$ for every point $s\in S$,
$h_{*}\mathcal{T}_{Y/S,s}=0$, $R^{1}gh_{*}\mathcal{T}_{Y/S}$ is
locally free of rank $g-1$ \cite[Corollary III.12.9]{Ha77} and so,
$H^{1}(Y,\mathcal{T}_{Y/S})\simeq H^{0}(S,R^{1}h_{*}\mathcal{T}_{Y/S})$
by the Leray spectral sequence. Replacing $S$ by an open subset,
we may assume that $R^{1}h_{*}\mathcal{T}_{Y/S}$ admits a nowhere
vanishing global section $\sigma$. Via the isomorphism $H^{1}(Y,\mathcal{T}_{Y/S})\simeq\mathrm{Ext}_{Y}^{1}(\mathcal{O}_{Y},\mathcal{T}_{Y/S})$,
we may interpret this section as the class of a non trivial extension
$0\rightarrow\mathcal{T}_{Y/S}\rightarrow\mathcal{E}\rightarrow\mathcal{O}_{Y}\rightarrow0$
of locally free sheaves over $Y$. The inclusion $\mathcal{T}_{Y/S}\rightarrow\mathcal{E}$
defines a section $D$ of the locally trivial $\mathbb{P}^{1}$-bundle
$\overline{\rho}:\overline{X}=\mathrm{Proj}(\mathrm{Sym}_{\mathcal{O}_{Y}}\mathcal{E}^{\vee})\rightarrow Y$
and the non vanishing of $\sigma$ guarantees that $D$ is the support
of an $S$-ample divisor. Indeed the $S$-ampleness of $D$ is equivalent
to the property that for every $s\in S$ the induced section $D_{s}$
of the $\mathbb{P}^{1}$ -bundle $\overline{\rho}_{s}:\overline{X}_{s}\rightarrow Y_{s}$
over the smooth projective curve $Y_{s}$ is ample. Since by construction,
$\overline{\rho}_{s}\mid_{\overline{X}_{s}\setminus D_{s}}:\overline{X}_{s}\setminus D_{s}\rightarrow Y_{s}$
is a nontrivial torsor under the line bundle $\mathrm{Spec}(\mathrm{Sym}\mathcal{T}_{Y_{s}}^{\vee})\rightarrow Y_{s}$,
it follows that $D_{s}$ intersects positively every section $D$
of $\overline{\rho}_{s}$ except maybe $D_{s}$ itself. On the other
hand, we have $(D_{s}^{2})=-\deg\mathcal{T}_{Y_{s}}=2g(Y_{s})-2>0$,
and so the ampleness of $D_{s}$ follows from the Nakai-Moishezon
criterion and the description of the cone effective cycles on an irrational
projective ruled surface given in \cite[Proposition 2.20-2.21]{Ha77}. 

Letting $X=\overline{X}\setminus D$, we obtain a smooth family 
\[
f=g\circ\overline{\rho}\mid_{X}:X\stackrel{\overline{\rho}\mid_{X}}{\longrightarrow}Y\stackrel{h}{\rightarrow}S
\]
where $\overline{\rho}\mid_{X}:X\rightarrow Y$ is nontrivial, locally
trivial, $\mathbb{A}^{1}$-bundle such that for every $s\in S$, $X_{s}$
is an affine surface with an $\mathbb{A}^{1}$-fibration $\rho_{s}:X_{s}\rightarrow Y_{s}$
of complete type. 
\end{example}
In contrast with the previous example, the following proposition shows
in particular that if the total space of a family of irrational $\mathbb{A}^{1}$-ruled
affine surfaces $f:X\rightarrow S$ has finite divisor class group,
then the induced $\mathbb{A}^{1}$-fibration on a general fiber of
$f:X\rightarrow S$ is of affine type. 
\begin{prop}
\label{prop:Ga-quotient-1} Let $X$ be a geometrically integral normal
variety with finite divisor class group $\mathrm{Cl}(X)$ and let
$f:X\rightarrow S$ be a dominant affine morphism to a normal variety
$S$ with the property that for a general closed point $s\in S$,
the fiber $X_{s}$ is irrational and $\mathbb{A}^{1}$-ruled, say
with unique $\mathbb{A}^{1}$-fibration $\pi_{s}:X_{s}\rightarrow C_{s}$.
Then there exists an effective $\mathbb{G}_{a,S}$-action on $X$
such that for a general closed point $s\in S$, the $\mathbb{A}^{1}$-fibration
$\pi_{s}:X_{s}\rightarrow C_{s}$ factors through the algebraic quotient
$\rho_{s}:X_{s}\rightarrow X_{s}/\!/\mathbb{G}_{a,s}=\mathrm{Spec}(\Gamma(X_{s},\mathcal{O}_{X_{s}})^{\mathbb{G}_{a,s}})$. \end{prop}
\begin{proof}
Let $f\mid_{X_{*}}=h\circ\rho:X_{*}\stackrel{\rho}{\longrightarrow}Y\stackrel{h}{\longrightarrow}S_{*}$
be as in Theorem \ref{thm:Mrc-fibration}. Since $\rho$ is an $\mathbb{A}^{1}$-fibration,
there exists an affine open subset $U\subset Y$ such that $\rho^{-1}(U)\simeq U\times\mathbb{A}^{1}$
as schemes over $U$. Since $\rho^{-1}(U)$ is affine, its complement
in $X$ is of pure codimension $1$, and the finiteness of $\mathrm{Cl}(X)$
implies that it is actually the support of an effective principal
divisor $\mathrm{div}_{X}(a)$ for some $a\in\Gamma(X,\mathcal{O}_{X})$.
Letting $\partial_{0}$ be the locally nilpotent derivation of $\Gamma(\rho^{-1}(U),\mathcal{O}_{X})\simeq\Gamma(X,\mathcal{O}_{X})_{a}$
corresponding to the $\mathbb{G}_{a,U}$-action by translations on
the second factor, the finite generation of $\Gamma(X,\mathcal{O}_{X})$
guarantees that for a suitably chosen $n\geq0$, $a^{n}\partial_{0}$
lifts to a locally nilpotent derivation $\partial$ of $\Gamma(X,\mathcal{O}_{X})$.
By construction, the restriction of $f$ to the dense open subset
$\rho^{-1}(U)$ of $X$ is invariant under the corresponding $\mathbb{G}_{a}$-action,
and so $f:X\rightarrow S$ is $\mathbb{G}_{a}$-invariant. For a general
closed point $s\in S$, the induced $\mathbb{G}_{a}$-action on $X_{s}$
is nontrivial, and its algebraic quotient $\rho_{s}:X_{s}\rightarrow X_{s}/\!/\mathbb{G}_{a}=\mathrm{Spec}(\Gamma(X_{s},\mathcal{O}_{X_{s}})^{\mathbb{G}_{a}})$
is a surjective $\mathbb{A}^{1}$-fibration over a normal affine curve
$X_{s}/\!/\mathbb{G}_{a}$. Since $C_{s}$ is irrational, the general
fibers of $\rho_{s}$ and $\pi_{s}$ must coincide. It follows that
$\pi_{s}$ is $\mathbb{G}_{a}$-invariant, whence factors through
$\rho_{s}$. 
\end{proof}

\section{Affine threefolds fibered in irrational $\mathbb{A}^{1}$-ruled surfaces}

In this section we consider in more detail the case of normal complex
affine threefolds $X$ admitting a fibration $f:X\rightarrow B$ by
irrational $\mathbb{A}^{1}$-ruled surfaces, over a smooth curve $B$.
We explain how to derive the variety $h:Y\rightarrow B$ for which
$f$ factors through an $\mathbb{A}^{1}$-fibration $\rho:X\rightarrow Y$
from a relative minimal model program applied to a suitable projective
model of $X$ over $B$. In the case where the divisor class group
of $X$ is finite, we provide a complete classification of such fibrations
in terms of additive group actions on $X$.

\subsection{$\mathbb{A}^{1}$-cylinders via relative Minimal Model Program }

\indent\newline\indent Let $X$ be a normal complex affine threefold
and let $f:X\rightarrow B$ be a flat morphism onto a smooth curve
$B$ with the property that a general closed fiber $X_{b}$ of $f$
is an irreducible irrational $\mathbb{A}^{1}$-ruled surface. We let
$\overline{f}:W\rightarrow B$ be a smooth projective model of $X$
over $B$ obtained from an arbitrary normal relative projective completion
$X\hookrightarrow\overline{X}$ of $X$ over $B$ by resolving the
singularities. We let $j:X\dashrightarrow W$ be the birational map
induced by the open immersion $X\hookrightarrow\overline{X}$. 

By applying a minimal model program for $W$ over $B$, we obtain
a sequence of birational $B$-maps 
\[
W=W_{0}\stackrel{\varphi_{1}}{\dashrightarrow}W_{1}\stackrel{\varphi_{2}}{\dashrightarrow}W_{2}\dashrightarrow\cdots\dashrightarrow W_{\ell-1}\stackrel{\varphi_{\ell}}{\dashrightarrow}W_{\ell}=W',
\]
between $B$-schemes $\overline{f}_{i}:W_{i}\rightarrow B$, where
$\varphi_{i}:W_{i}\dashrightarrow W_{i+1}$ is either a divisorial
contraction or a flip, and the rightmost variety $W'$ is the output
of a minimal model program over $B$. The hypotheses imply that $W'$
has the structure of a Mori conic bundle $\overline{\rho}:W'\rightarrow Y$
over a $B$-scheme $h:Y\rightarrow B$ corresponding to the contraction
of an extremal ray of $\overline{\mathrm{NE}}(W'/B)$. Indeed, a general
fiber of $\overline{f}$ being a birationaly ruled projective surface,
the output $W'$ is not a minimal model of $W$ over $B$. So $W'$
is either a Moric conic bundle over a $B$-scheme $Y$ of dimension
$2$ or a del Pezzo fibration over $B$, the second case being excluded
by the fact that the general fibers of $\overline{f}$ are irrational. 
\begin{prop}
\label{prop:MMP-output}The induced map $\rho=\overline{\rho}\mid_{X}:X\dashrightarrow Y$
is a rational $\mathbb{A}^{1}$-fibration.\end{prop}
\begin{proof}
Since a general closed fiber $X_{b}$ is a normal affine surface with
an $\mathbb{A}^{1}$-fibration $\pi_{b}:X_{b}\rightarrow C_{b}$ over
a certain irrational smooth curve $C_{b}$, it follows that there
exists a unique maximal affine open subset $U_{b}$ of $C_{b}$ such
that $\pi_{b}^{-1}(U_{b})\simeq U_{b}\times\mathbb{A}^{1}$ and such
that the rational map $j_{b}:\pi_{b}^{-1}(U_{b})\dashrightarrow W_{b}$
induced by $j$ is regular, inducing an isomorphism between $\pi_{b}^{-1}(U_{b})$
and its image. Each step $\varphi_{i}:W_{i}\dashrightarrow W_{i+1}$
consists of either a flip whose flipping and flipped curves are contained
in fibers of $\overline{f}_{i}:W_{i}\rightarrow B$ and $\overline{f}_{i+1}:W_{i+1}\rightarrow B$
respectively, or a divisorial contraction whose exceptional divisor
is contained in a fiber of $\overline{f}_{i}:W_{i}\rightarrow B$,
or a divisorial contraction whose exceptional divisor intersects a
general fiber of $\overline{f}_{i}:W_{i}\rightarrow B$. Clearly,
a general closed fiber of $\overline{f}_{i}:W_{i}\rightarrow B$ is
not affected by the first two types of birational maps. On the other
hand, if $\varphi_{i}:W_{i}\rightarrow W_{i+1}$ is the contraction
of a divisor $E_{i}\subset W_{i}$ which dominates $B$, then a general
fiber of $\varphi_{i}\mid_{E_{i}}$ is a smooth proper rational curve.
The intersection of $E_{i}$ with a general closed fiber $W_{i,b}$
of $\overline{f}_{i}$ thus consists of proper rational curves, and
its intersection with the image of the maximal affine cylinder like
open subset $\pi_{b}^{-1}(U_{b})$ of $X_{b}$ is either empty or
composed of affine rational curves. Since $U_{b}$ is an irrational
curve, it follows that each irreducible component of $E_{i}\cap(\pi_{b}^{-1}(U_{b}))$
is contained in a fiber of $\pi_{b}$. This implies that there exists
an open subset $U_{b,0}$ of $U_{b}$ with the property that for every
$i=1,\ldots,\ell$, the restriction of $\varphi_{i}\circ\cdots\circ\varphi_{1}\circ j$
to $\pi_{b}^{-1}(U_{b,0})\subset X_{b}$ is an isomorphism onto its
image in $W_{i,b}$. A general fiber of $\overline{\rho}:W'\rightarrow Y$
over a closed point $y\in Y$ being a smooth proper rational curve,
its intersection with $\pi_{h(y)}^{-1}(U_{h(y),0})$ viewed as an
open subset of $W'_{h(y)}$, is thus either empty or equal to a fiber
of $\pi_{h(y)}$. So by virtue of \cite{KM78}, there exists an open
subset $V$ of $X$ on which $\overline{\rho}$ restricts to an $\mathbb{A}^{1}$-fibration
$\overline{\rho}\mid_{V}:V\rightarrow Y$. \end{proof}
\begin{cor}
\label{cor:Threefold-Bir-model}Let $X$ be a normal complex affine
threefold $X$ equipped with a morphism $f:X\rightarrow B$ onto a
smooth curve $B$ whose general closed fibers are irrational $\mathbb{A}^{1}$-ruled
surfaces. Then $X$ is birationaly equivalent to the product of $\mathbb{P}^{1}$
with a family $h_{0}:\mathcal{C}_{0}\rightarrow B_{0}$ of smooth
projective curves of genus $g\geq1$ over an open subset $B_{0}\subset B$. \end{cor}
\begin{proof}
By the previous Proposition, $X$ has the structure of a rational
$\mathbb{A}^{1}$-fibration $\rho:X\dashrightarrow Y$ over a $2$-dimensional
normal proper $B$-scheme $h:Y\rightarrow B$. In particular, $X$
is birational to $Y\times\mathbb{P}^{1}$. On the other hand, for
a general closed point $b\in B$, the curve $Y_{b}$ is birational
to the base $C_{b}$ of the unique $\mathbb{A}^{1}$-fibration $\pi_{b}:X_{b}\rightarrow C_{b}$
on the irrational affine surface $X_{b}$. Letting $\sigma:\tilde{Y}\rightarrow Y$
be a desingularization of $Y$, there exists an open subset $B_{0}$
of $B$ over which the composition $h\circ\sigma:\tilde{Y}\rightarrow Y$
restricts to a smooth family $h_{0}:\mathcal{C}_{0}\rightarrow B_{0}$
of projective curves of a certain genus $g\geq1$. By construction,
$X$ is birational to $\mathcal{C}_{0}\times\mathbb{P}^{1}$. \end{proof}
\begin{rem}
Example \ref{ex:fibration-from-moduli} above shows conversely that
for every smooth family $h:\mathcal{C}\rightarrow B$ of projective
curves of genus $g\geq2$, there exists a smooth $\mathbb{A}^{1}$-ruled
affine threefold $X$ birationaly equivalent to $\mathcal{C}\times\mathbb{P}^{1}$.
Actually, in the setting of the previous Corollary \ref{cor:Threefold-Bir-model},
if we assume further that a general fiber of $f:X\rightarrow B$ carries
an $\mathbb{A}^{1}$-fibration $\pi_{b}:X_{b}\rightarrow C_{b}$ over
a smooth curve $C_{b}$ whose smooth projective model has genus $g\geq2$,
then there exists a uniquely determined family $h:\mathcal{C}\rightarrow B$
of proper stable curves of genus $g$ such that $X$ is birationaly
equivalent to $\mathcal{C}\times\mathbb{P}^{1}$: indeed, the moduli
stack $\overline{\mathcal{M}}_{g}$ of stable curves of genus $g\geq2$
being proper and separated, the smooth family $h_{0}:\mathcal{C}_{0}\rightarrow B_{0}$
extends in a unique way to a family $h:\mathcal{C}\rightarrow B$
of stable curves of genus $g$.
\end{rem}

\subsection{Factorial threefolds }
\begin{prop}
\label{prop:Ga-quotient-2}Let $X$ be a normal affine threefold with
finite divisor class group $\mathrm{Cl}(X)$ and let $f:X\rightarrow B$
be a morphism onto a smooth curve $B$ whose general closed fibers
are irrational $\mathbb{A}^{1}$-ruled surfaces. Then there exists
a factorisation $f=h\circ\rho:X\rightarrow Y\rightarrow B$ where
$\rho:X\rightarrow Y$ is the algebraic quotient morphism of an effective
$\mathbb{G}_{a,B}$-action on $X$. In particular, a general fiber
of $f$ admits an $\mathbb{A}^{1}$-fibration of affine type. \end{prop}
\begin{proof}
By virtue of Proposition \ref{prop:Ga-quotient-1}, there exist an
effective $\mathbb{G}_{a,B}$-action on $X$ such that for a general
closed point $b\in B$, the $\mathbb{A}^{1}$-fibration $\pi_{b}:X_{b}\rightarrow C_{b}$
on $X_{b}$ factors through the algebraic quotient $\rho_{b}:X_{b}\rightarrow X_{b}/\!/\mathbb{G}_{a,b}=\mathrm{Spec}(\Gamma(X_{b},\mathcal{O}_{X_{b}})^{\mathbb{G}_{a,b}})$.
Since $X$ is a threefold, the ring of invariants $\Gamma(X,\mathcal{O}_{X})^{\mathbb{G}_{a,B}}$
is finitely generated \cite{Zar54}. The quotient morphism $\rho:X\rightarrow Y=\mathrm{Spec}(\Gamma(X,\mathcal{O}_{X})^{\mathbb{G}_{a,B}})$
is an $\mathbb{A}^{1}$-fibration, and since $Y$ is a categorical
quotient in the category of algebraic varieties, the invariant morphism
$f:X\rightarrow B$ factors through $\rho$. \end{proof}
\begin{cor}
\label{cor:A3-irrational} Let $f:\mathbb{A}^{3}\rightarrow B$ be
a morphism onto a smooth curve $B$ with irrational $\mathbb{A}^{1}$-ruled
general fibers. Then $B$ is isomorphic to either $\mathbb{P}^{1}$
or $\mathbb{A}^{1}$ and there exists a factorization $f=h\circ\rho:\mathbb{A}^{3}\rightarrow\mathbb{A}^{2}\rightarrow B$,
where $\rho:\mathbb{A}^{3}\rightarrow\mathbb{A}^{2}$ is the quotient
morphism of an effective $\mathbb{G}_{a,B}$-action on $\mathbb{A}^{3}$. \end{cor}
\begin{proof}
Since $B$ is dominated by a general line in $\mathbb{A}^{3}$, it
is necessarily isomorphic to $\mathbb{P}^{1}$ or $\mathbb{A}^{1}$.
The second assertion follows from Proposition \ref{prop:Ga-quotient-2}
and the fact that the algebraic quotient of every nontrivial $\mathbb{G}_{a}$-action
on $\mathbb{A}^{3}$ is isomorphic to $\mathbb{A}^{2}$ \cite{Mi84}. \end{proof}
\begin{example}
In Corollary \ref{cor:A3-irrational} above, the base curve $B$ need
not be affine. For instance, the morphism 
\[
f:\mathbb{A}^{3}=\mathrm{Spec}(\mathbb{C}[x,y,z])\longrightarrow\mathbb{P}^{1},(x,y,z)\mapsto[((xz-y^{2})x^{2}+1:(xz-y^{2})^{3}]
\]
defines a family whose general member is isomorphic to the product
$C_{\lambda}\times\mathbb{A}^{1}$ where $C_{\lambda}\subset\mathbb{A}^{2}=\mathrm{Spec}(\mathbb{C}[xz-y^{2},x])$
is the affine elliptic curve with equation $(xz-y^{2})^{3}+\lambda((xz-y^{2})x^{2}+1)=0$.
The subring $\mathbb{C}[xz-y^{2},x]$ of $\mathbb{C}[x,y,z]$ coincides
with the ring of invariants of the $\mathbb{G}_{a}$-action associated
with the locally nilpotent $\mathbb{C}[x]$-derivation $x\partial_{y}+2y\partial_{z}$
and $f$ is the composition of the quotient morphism $\rho:\mathbb{A}^{3}\rightarrow\mathbb{A}^{2}=\mathbb{A}^{3}/\!/\mathbb{G}_{a}=\mathrm{Spec}(\mathbb{C}[u,v])$,
$(x,y,z)\mapsto(xz-y^{2},x)$ and the morphism $h:\mathbb{A}^{2}=\mathrm{Spec}(\mathbb{C}[u,v])\rightarrow\mathbb{P}^{1}$,
$\left(u,v\right)\mapsto[uv^{2}+1:u^{3}]$.
\end{example}
Corollary \ref{cor:A3-irrational} above implies in particular that
a general fiber of a regular function $f:\mathbb{A}^{3}\rightarrow\mathbb{A}^{1}$
cannot be an irrational surface equipped with an $\mathbb{A}^{1}$-fibration
of complete type only. In contrast, regular functions $f:\mathbb{A}^{3}\rightarrow\mathbb{A}^{1}$
whose general fibers are rational and equipped with $\mathbb{A}^{1}$-fibrations
of complete type only do exist, as illustrated by the following example.
\begin{example}
Let $f=x^{3}-y^{3}+z(z+1)\in\mathbb{C}[x,y,z]$ and let $f:\mathbb{A}^{3}=\mathrm{Spec}(\mathbb{C}[x,y,z])\rightarrow\mathbb{A}^{1}=\mathrm{Spec}(\mathbb{C}[\lambda])$
be the corresponding morphism. The closure $\overline{S}_{\lambda}$
in $\mathbb{P}^{3}=\mathrm{Proj}(\mathbb{C}[x,y,z,t])$ of a general
fiber $S_{\lambda}=f^{*}(\lambda)$ of $f$ is a smooth cubic surface
which intersects the hyperplane $H_{\infty}=\left\{ t=0\right\} $
along the union $B_{\lambda}$ of three lines meeting at the Eckardt
point $p=\left[0:0:1:0\right]$. Thus $S_{\lambda}$ is rational and
a direct computation reveals that $\kappa(S_{\lambda})=-\infty$.
So by virtue \cite{MS80}, $S_{\lambda}$ admits an $\mathbb{A}^{1}$-fibration
$\pi_{\lambda}:S_{\lambda}\rightarrow C_{\lambda}$ over a smooth
rational curve $C_{\lambda}$. If $C_{\lambda}$ was affine, then
there would exist a non trivial $\mathbb{G}_{a}$-action on $S_{\lambda}$
having the general fibers of $\pi_{\lambda}$ as general orbits. But
it is straightforward to check that every automorphism of $S_{\lambda}$
considered as a birational self-map of $\overline{S}_{\lambda}$ is
in fact a biregular automorphism of $\overline{S}_{\lambda}$ preserving
the boundary $B_{\lambda}$. So the automorphism group of $S_{\lambda}$
injects into the group $\mathrm{Aut}(\overline{S}_{\lambda},B_{\lambda})$
of automorphisms of the pair $(\overline{S}_{\lambda},B_{\lambda})$.
The latter being a finite group, we conclude that no such $\mathbb{G}_{a}$-action
exists, and hence that $S_{\lambda}$ only admits $\mathbb{A}^{1}$-fibrations
of complete type. An $\mathbb{A}^{1}$-fibration $\pi_{\lambda}:S_{\lambda}\rightarrow\mathbb{P}^{1}$
can be obtained as follows: letting $B_{\lambda}=L_{1}\cup L_{2}\cup L_{3}$,
$L_{1}$ is a member of a $6$-tuple of pairwise disjoint lines whose
simultaneous contraction realizes $\overline{S}_{\lambda}$ as a blow-up
$\sigma:\overline{S}_{\lambda}\rightarrow\mathbb{P}^{2}$ of $\mathbb{P}^{2}$
in such a way that $\sigma(L_{2})$ and $\sigma(L_{3})$ are respectively
a smooth conic and its tangent line at the point $p=\sigma(L_{1})$.
The birational transform $\overline{\pi}_{\lambda}:\overline{S}_{\lambda}\dashrightarrow\mathbb{P}^{1}$
on $\overline{S}_{\lambda}$ of the pencil generated by $\sigma(L_{2})$
and $2\sigma(L_{3})$ restricts to an $\mathbb{A}^{1}$-fibration
$\pi_{\lambda}:S_{\lambda}\rightarrow\mathbb{P}^{1}$ with two degenerate
fibers: an irreducible one, of multiplicity two, consisting of the
intersection with $S_{\lambda}$ of the unique exceptional divisor
of $\sigma$ whose center is supported on $\sigma(L_{3})\setminus\{p\}$,
and a smooth one consisting of the intersection with $S_{\lambda}$
of the four exceptional divisors of $\sigma$ with centers supported
on $\sigma(L_{2})\setminus\left\{ p\right\} $. 
\end{example}
\bibliographystyle{amsplain}

\end{document}